\documentclass[12pt, reqno]{amsart}
\usepackage{amsmath, amsthm, amscd, amsfonts, amssymb, graphicx, color}
\usepackage[bookmarksnumbered, colorlinks, plainpages]{hyperref}
\hypersetup{colorlinks=true,linkcolor=red, anchorcolor=green, citecolor=cyan, urlcolor=red, filecolor=magenta, pdftoolbar=true}

\newtheorem{theorem}{Theorem}[section]
\newtheorem{lemma}[theorem]{Lemma}

\newtheorem{corollary}[theorem]{Corollary}
\theoremstyle{definition}
\newtheorem{definition}[theorem]{Definition}

\theoremstyle{remark}
\newtheorem{remark}[theorem]{Remark}

\def\log{\mathrm{log}}

\def\det{\mathrm{det}}

\makeatletter

\newcommand{\Rmnum}[1]{\expandafter\@slowromancap\romannumeral #1@}
\makeatother
\newcommand{\un}{{\rm 1}\mkern -4mu{{\rm l}}}

\begin{document}
\setcounter{page}{1}

\title[logarithmic  submajorisations inequalities]{logarithmic submajorisations
 inequalities for operators in a finite von Neumann algebra}

\author[C. Yan ]{Cheng Yan}
\address{ College of Mathematics and Systems Sciences, Xinjiang
University,Urumqi 830046, China.}
\email{\textcolor[rgb]{0.00,0.00,0.84}{yanchengggg@163.com}}

\author[Y. Han]{Yazhou Han$^{*}$
}
\thanks{$^{*}$ Corresponding Author}
\address{College of Mathematics and Systems Science, Xinjiang
University, Urumqi 830046, China}
\email{\textcolor[rgb]{0.00,0.00,0.84}{hanyazhou@foxmail.com}}

\subjclass[2010]{ Primary 46L52; Secondary 47A63 }

\keywords{von Neumann algebra; logarithmic  submajorisations; Fuglede-Kadison determinant inequality}

\begin{abstract} The aim of this paper is to study the logarithmic  submajorisations
inequalities for operators in a finite von Neumann algebra.
Firstly, some logarithmic  submajorisations inequalities due to Garg and  Aulja are
extended to the case of  operators in a finite von Neumann algebra.
As an application, we get some new Fuglede-Kadison  determinant
inequalities of operators in that circumstance.
Secondly, we improve and generalize to the setting of finite von Neumann algebras,
a generalized H\"{o}lder type generalized singular numbers inequality.
\end{abstract}
\maketitle

\section{Introduction}
The views  of determinants  and majorization inequalities in mathematical analysis are
based on the convexity(or concavity) of function and the  singular values equations and inequalities.
For instance,  Weyl's Theorem in matrix analysis involves that idea in the proof.
Using singular values  inequalities and log convexity of determinants inequalities, Key Fan
obtained two different forms of generalizations of Weyl's theorem in \cite{Fan1949,Fan1950}.
The singular values of $A\in \mathbb{M}_n$, the space of $n\times n$ complex matrices, i.e.,
the eigenvalues of the matrix $|A|:=(A^*A)^{\frac{1}{2}}$, enumerated in decreasing order,
will be denoted by  $s_j(A)$, $ j=1,~2, \ldots,~n.$

Rotfel'd \cite{R1969} proved  a  determinants inequality:
\begin{equation}\label{ineq 1.1}
\det(I_n + r|A + B|^p) \leq  \det(I_n + r|A|^p) \det(I_n + r|B|^p),
\end{equation}
where $A,~B\in M_n$, $r > 0$, $0 < p \leq 1$ and  $I_n$ is the identity matrix.
Subsequently, Garg-Aujla \cite{GA2017} proposed the following
singular value inequality for matrices:
\begin{equation}\label{ineq 1.002}
\Pi_{j=1}^k s_j(I_n+f(|A+B|))\leq \Pi_{j=1}^k s_j(I_n+f(|A|))s_j(I_n+f(|B|)),
\end{equation}
where $1\leq k \leq n$ and $A,~B\in \mathbb{M}_n$ and
$f: [0,+\infty) \rightarrow [0,+\infty)$
is an operator concave function with $f(0)=0$.
This is a refinement of the inequality (\ref{ineq 1.1}).
They also proved that for $A,~B \in \mathbb{M}_n$ and $1 \leq r \leq 2$,
\begin{equation}\label{ineq 1.003}
\Pi_{j=1}^k s_j( |A+B|^r)\leq \Pi_{j=1}^k s_j(I_n+ |A|^r)s_j(I_n+ |B|^r)
\end{equation}
holds for $1 \leq k \leq n$.
Very recently,  a new proof for inequalities (\ref{ineq 1.002}) and
(\ref{ineq 1.003}) were gave by Zhao in \cite{Z2018}. Meanwhile, they showed that inequality (\ref{ineq 1.002}) also holds
when $f$ is a nonnegative concave function.
Liu-Poon-Wang \cite{Wang2017} proved  a generalized H\"{o}lder type eigenvalue inequality:
\begin{equation}\label{ineq 1.004}\Pi_{j=1}^k (1-s_j(|A_1\cdots A_m|)^r)\geq \Pi_{j=1}^k\Pi_{i=1}^m(1-s_j(|A_i|)^{rp_i})^{\frac{1}{p_i}},\end{equation}
where   $A_1,\ldots,A_m$ are $n\times n$
contractive matrices and $p_1,\ldots,p_m>0$, with $\sum_{i=1}^m \frac{1}{p_i}=1$,
for each $k = 1, 2,\ldots , n$ and $r\geq 1$.

With the help of generalized singular numbers' method,  Fack \cite{Fack1983} and Fack-Kosaki \cite{FK1986}
gave the inequality (\ref{ineq 1.1}) and (\ref{ineq 1.002})
for operators in a semi-finite von Neumann algebra.
It is our intention in this paper to indicate that
the logarithmic submajorisations inequalities
of  Garg-Aulja  may be extended to the general setting of
operators in a finite von Neumann algebra.
And we prove a generalized H\"{o}lder type generalized singular numbers inequality
in the case of operators in a finite von Neumann algebra.

 This article is organized as follows.
 In section 2, we setup the background for our discussion. Along with setting up notation, we present a primer on the theory of von Neumann algebras, two kinds of generalized singular numbers, Fuglede-Kadison determinant of measure operators affiliated with  a finite von Neumann algebras.
In section 3, we collect some   lemmas of  logarithmic submajorisations inequalities under the case of the monotone concave function on
$[0,+\infty)$ and extend Garg and  Aulja's  results to operator case.
The crux of the discussion is in the final section
where some basic equations and inequalities of two kinds of generalized singular numbers are proved.  And we propose a generalized H\"{o}lder type generalized singular numbers inequality in this section.

\section{Preliminaries}
Suppose that $\mathcal{H}$ is a separable Hilbert space over the field $\mathbb{C}$ and
$\mathbb{I}$ is the identity operator in $\mathcal{H}$. We will denote by
  $\mathcal{B}(\mathcal{H})$ the $*$-algebra of all linear bounded operators in $\mathcal{H}$.
Let $\mathcal{M}$ be a $*$-subalgebra of
$\mathcal{B}(\mathcal{H})$ containing the identity operator $\mathbb{I}$.
Then $\mathcal{M}$ is called a von Neumann algebra
if $\mathcal{M}$ is weak* operator closed.
Let $\mathcal{M}$  be a finite von
Neumann algebra, with a finite normal faithful trace $\tau$,  acting on the separable Hilbert space $\mathcal{H}$,
 and $\mathcal{M}_+$ its positive part.  We refer to \cite{FK1986, PX2003} for noncommutative integration.

Let $x$ be a closed densely defined operator and $x = u|x|$ its polar decomposition,
where $|x|=(x^*x)^{\frac12}$ and $u$  is a partial isometry.
Then $x$ is affiliated with $\mathcal{M}$ iff $u\in\mathcal{M}$ and $|x|$ is affiliated with $\mathcal{M}$.
 For convenience, we assume $\tau(\un)=1$ in the following.

For  $x\in \mathcal{M}$, we define the generalized singular numbers by
\[\mu_t(x)=\inf \{\lambda>0: \tau(e_\lambda(|x|)\leq t) \},~t> 0,\]
where the operators $e_s(|x|)$ are the spectral projection of $|x|$.
We denote simply
by $\mu(x)$ the function $t\rightarrow\mu_t(x)$.
If we
consider the algebra $\mathcal{M}=L^\infty([0, 1])$ of all Lebesgue measurable essentially bounded
functions on $[0, 1]$. 
For $f\in L^\infty([0, 1])$, the
decreasing rearrangement $f^*$ of the function $f$ is given by
\[f^*(t)=\inf\{s\in \mathbb{R}: m(\{h\in[0,1]: |f(h)|>s\})\leq t\}, 0<t<1.\]
Then $\mu_t(f)=f^*(t)$.

As a matter of convenience, we state some properties
of generalized singular numbers as follows without proof(see \cite{FK1986}).
The function $t\rightarrow \mu_t(x)$ is a nonincreasing right-continuous function
on $(0, \infty)$ and
\begin{equation}\label{equa 2.1}
\mu_t(x^*x)=\mu_t(xx^*)~\mbox{and}~\mu_t(uxv)\leq\|u\|\mu_t(x)\|v\|,
\end{equation}
where $x, u , v \in \mathcal{M}$.
If $f$ is a continuous increasing function  on $[0, \infty)$ with $f(0) \geq 0$, then
\begin{equation}\label{pre 1.2}
\mu_t(f(x))=f(\mu_t(x))
  \end{equation}
and
\begin{equation*}\label{pre 1.002}
\tau(f(x))=\int_0^{\tau(\un)} f(\mu_t(x))dt.
\end{equation*}
See \cite{FK1986} for basic properties and detailed information
on generalized singular number of $x$.

For  $x\in \mathcal{M}$ we now introduce the nonincreasing left-continuous function
\[\mu_t^l(x)=\inf \{\lambda>0: \tau(e_\lambda(|x|)< t) \},~t> 0.\]
Except for continuity, $\mu^l(x)$ and $\mu(x)$ have many similar properties.
 See \cite{O1970} and \cite{O19701} for basic properties and detailed information
of this  nonincreasing left-continuous function.

For $x\in \mathcal{M}$, we define
\begin{equation}\label{033}
\Lambda_t(x)=\exp{(\int_0^t\log \mu_s(x)ds)},~t>0.
\end{equation}
From the definition of $\Lambda_t(x)$ and the properties of
$\mu_t(x)$, we obtain
\begin{equation*}\label{03}
 \Lambda_t(x)=\Lambda_t(x^*)=\Lambda_t(|x|),~t>0
\end{equation*}
and
\begin{equation}\label{04}
 \Lambda_t(x^\alpha)=\Lambda_t(x )^\alpha,~t>0,~\mbox{if}~\alpha>0~\mbox{and}~x>0.
\end{equation}
Moreover, it follows from \cite[Theorem 4.2]{FK1986} that
\begin{equation}\label{05}
\Lambda_t(xy)\leq\Lambda_t(x)\Lambda_t(y),~t>0
\end{equation}
holds for all $x,~y\in\mathcal{M}$.
\begin{definition}
Let  $x,~y\in  \mathcal{M}$. We say that
$x$ is logarithmically submajorized by $y$ and write $x\prec_{wlog} y$
if and only if
$$
\Lambda_t(x) \leq \Lambda_t(y)~\mbox{for all}~t\geq0.
$$
\end{definition}

\begin{definition} Let $\mathcal{M}$  be a finite von
Neumann algebra acting on a separable Hilbert space $\mathcal{H}$
with a normal faithful finite trace $\tau$.
For $x\in\mathcal{M}$, the Fuglede-Kadison determinant of $x$
is defined by $\Delta(x) = \exp\tau(\log |x|)$ if $|x|$ is invertible;
and otherwise, the Fuglede-Kadison determinant $\Delta(x)=\inf\Delta(|x| + \varepsilon \un)$,
the infimum takes over all scalars $\varepsilon>0$.
\end{definition}

Furthermore, let $f:\mathbb{R}^+\rightarrow\mathbb{R}^+$
be a continuous increasing
function with $f(0)=1$. If $x\in \mathcal{M}_+$,
then $t\rightarrow\log f(t)$ is a continuous increasing function with $\log f(0)=0$.
From \cite[Lemma 2.5]{FK1986} and the definition of determinant, we have
\begin{equation*}\label{pre 1.03}
\Delta(f(x))=\exp\int_0^{\tau(\un)}\log
(f(\mu_t(x)))dt=\Lambda_{\tau(\un)}(f(x)).
\end{equation*}

 See \cite{A1967, B1986, DDSZ2020} for basic properties and detailed information
on Fuglede-Kadison determinant and logarithmic  submajorisations of $x\in\mathcal{M}$.

Let $\mathbb{M}_2(\mathcal{M})$ denote the linear space of
$2\times 2$ matrices
$$
x=\left(\begin{array}{cccc}x_{11}&x_{12}\\
x_{21}&x_{22}
 \end{array}\right)
 $$
with entries $x_{ij}\in\mathcal{M},~ i,~j=1,~2$.
Let $\mathcal{H}^2=\mathcal{H}\oplus\mathcal{H}$, then $\mathbb{M}_2(\mathcal{M})$
is a von Neumann algebra
on the Hilbert space $\mathcal{H}^2.$  For $x\in \mathbb{M}_2(\mathcal{M})$, we define
$\tau_2(x)=\sum_{i=1}^2\tau(x_{ii})$. Then
$\tau_2$ is a normal faithful finite trace on $\mathbb{M}_2(\mathcal{M})$.
The direct sum of operators $x_1,~x_2\in \mathcal{M}$, denoted by
 $x_1\oplus x_2$, is the block-diagonal operator  matrix defined on $\mathcal{H}^2$ by
$$x_1\oplus x_2=\left(\begin{array}{cccc}x_{1}&0\\
0&x_{2}
 \end{array}\right).$$

\section{Some logarithmic submajorisations inequalities}

To achieve our main results, we state for easy reference the following fact, obtained
from  \cite{DD1992, MKX2019}, which will be applied below.

\begin{lemma}\label{2.4DD1992}
Let $x,~y\in \mathcal{M}$.
\begin{enumerate}
\item If $0<x\in \mathcal{M}$, then
$$\mu_t(\un+x)=1+\mu_t(x).$$
\item Let $a,~b \in \mathcal{M}$ be two positive operators. Then the matrix
$
\left(
   \begin{array}{cc}
     a & x \\
     x^* & b \\
   \end{array}
 \right)\in \mathbb{M}_2(\mathcal{M})
$
is a positive semidefinite operator if and only if
 $x=a^{\frac{1}{2}}wb^{\frac{1}{2}}$ for some
contraction $w$.
\end{enumerate}
\end{lemma}

\begin{lemma}\label{lemma 6}
Let $x,~y\in \mathcal{M}$ be two positive operators. Then
\begin{eqnarray*}\label{Inequality 1}
\Lambda_t(y^px^py^p)\leq \Lambda_t((yxy)^p),~0\leq p\leq 1
\end{eqnarray*}
and
\begin{eqnarray*}\label{Inequality 2}
\Lambda_t(y^px^py^p)\geq \Lambda_t((yxy)^p),~ p\geq 1.
\end{eqnarray*}
\end{lemma}
\begin{proof}
It follows from \cite[Proposition 1.11]{B1986} and \cite[Lemma 2.5]{Han2016}.
\end{proof}

In the following theorem, we extend inequality (\ref{ineq 1.003}) to the case of operators of finite von Neumann algebra.
\begin{theorem}
Let $x,~y \in \mathcal{M}$ . Then
\[\Lambda_t(|x+y|^r)\leq \Lambda_t(\un +|x|^r)\Lambda_t(\un +|y|^r) \]
holds for $t\geq 0$ and $1 \leq r \leq 2$. Moreover, we have
\[\Delta(|x+y|^r)\leq \Delta(\un +|x|^r)\Delta(\un +|y|^r) \]
holds for  $1 \leq r \leq 2$.
\end{theorem}
\begin{proof} Since
\[\left(
    \begin{array}{cc}
      \un+xx^* & x+y \\
      (x+y)^* & \un+y^*y \\
    \end{array}
  \right)
  =\left(
    \begin{array}{cc}
      \un & x \\
      y^* & \un \\
    \end{array}
  \right)
  \left(
    \begin{array}{cc}
      \un & y \\
      x^* & \un \\
    \end{array}
  \right)\geq 0,
\]
by Lemma \ref{2.4DD1992}~(2), there exists a contraction $w$ with
\begin{eqnarray*}
x+y=(\un+xx^*)^{\frac{1}{2}}w(\un+y^*y)^{\frac{1}{2}}
=(\un+|x^*|^2)^{\frac{1}{2}}w(\un+|y|^2)^{\frac{1}{2}}.
\end{eqnarray*}
Thus
\[|x+y|^{2r}=[(\un+|y|^2)^{\frac{1}{2}}w^*(\un+|x^*|^2)w(\un+|y|^2)^{\frac{1}{2}}]^r,~1 \leq r \leq 2.\]
Together the above equality with Lemma \ref{lemma 6}
and (\ref{05}), we obtain
\begin{equation} \label{inequlity-2r}
\begin{array}{lll}
\Lambda_t(|x+y|^{2r})
&\leq & \Lambda_t((\un+|y|^2)^{\frac{r}{2}}(w^*(\un+|x^*|^2) w)^r(\un+|y|^2)^{\frac{r}{2}})\\
&\leq & \Lambda_t((\un+|y|^2)^{\frac{r}{2}})
\Lambda_t((w^*(\un+|x^*|^2) w)^r)
\Lambda_t((\un+|y|^2)^{\frac{r}{2}}))\\
&= &\Lambda_t((\un+|y|^2)^r)\Lambda_t((w^*(\un+|x^*|^2) w)^r),
\end{array}
\end{equation}
for $ t\geq 0$.
Since $w$ is a contraction, then
\[\mu_t(w^*(\un+|x^*|^2) w)\leq \mu_t(\un+|x^*|^2),~t>0.\]
Inequalities above and (\ref{inequlity-2r}) imply
\begin{equation}\label{det-n}
\begin{array}{lll}
\Lambda_t(|x+y|^{2r})
&\leq &\Lambda_t(\un+|y|^2)^r \Lambda_t(\un+|x^*|^2)^r\\
&=&\Lambda_t(\un+|y|^2)^r \Lambda_t(\un+|x|^2)^r,~t>0,
\end{array}
\end{equation}
the above equality holds due to the unitarily equivalent of $|x|^2$ and $|x^*|^2$.
When $r = 1$, by inequality above, we have
\begin{equation}\label{det-2}
\begin{array}{lll}
\Lambda_t(|x+y|^{2})
&\leq &\Lambda_t(\un+|y|^2) \Lambda_t(\un+|x|^2),~t\geq 0.
\end{array}
\end{equation}

On the other hand, $f(t) = t^{\frac{r}{2}} (1 \leq r < 2)$ is a concave function on $[0,+\infty)$, then
$f(s)+ f(t) \leq  f(s+ t)$, for $s,~t \in [0,+\infty)$. It follows that
\begin{equation}\label{mu-r}
(1+\mu_t(|x|)^2)^{\frac{r}{2}}\leq 1+ \mu_t(|x|)^r,~x\in \mathcal{M}~\text{and}~t\geq 0.
\end{equation}
Combining (\ref{pre 1.2}), (\ref{033}), (\ref{04}), (\ref{det-n})
 with (\ref{mu-r}) and noting that
$\mu_t(\un + |x|^2) = 1 + \mu_t(|x|)^2$
$(t\geq 0)$, for $x \in \mathcal{M}$, we get
\[\Lambda_t(|x+y|^{2r})\leq \Lambda_t(\un+|y|^r)^2 \Lambda_t(\un+|x|^r)^2\]
which means that
\begin{equation}\label{det-r}\Lambda_t(|x+y|^{r})\leq \Lambda_t(\un+|y|^r)\Lambda_t(\un+|x|^r),~t\geq 0~\text{and}~1 \leq r < 2.\end{equation}

Thus, the result follows from inequalities (\ref{det-2}) and (\ref{det-r}).
This completes the proof.
\end{proof}

\section{A generalized H\"{o}lder type eigenvalue inequality}

To prove the finial theorem, we start with the following equations of the two kinds of sigular numbers.
\begin{lemma}\label{lemma 10}
 Suppose $x \in \mathcal{M}$ is a contractive operator.
 Then
\[\mu_s(\un-|x|)=1-\mu_{1-s}^l(|x|),~0< s< 1\]
and \[\mu_s^l(\un-|x|)=1-\mu_{1-s}(|x|),~0< s< 1.\]
\end{lemma}
\begin{proof} Let
$$
K=\left\{x:\begin{array}{l}
             x=\sum_{k=1}^nc_ke_k, c_k\in\mathbb{C},~\text{for any positive integer}~ n \\
e_k\in P(\mathcal{M}), e_k\bot e_j,  if ~k\neq j,
\tau(e_k)<\infty ,~j, k=1, 2, \ldots, n
 \end{array}
\right\}.
$$
Since $K$ is dense in $\mathcal{M}$ with the operator norm,
it is sufficient to show the lemma holds for
  $x=\sum_{k=1}^Nc_ke_k\in K$. It is clear that $|x|=\sum_{k=1}^N|c_k|e_k$.
Without loss of generality, we suppose $1\geq |c_1|>|c_2|>\cdots>|c_N|$.
 Let $d_j=\sum_{k=1}^j\tau(e_k), 1\leq j\leq N$, $d_0=0$ and $\chi$ denote the indicative function.  Then
$$
\mu_s(x)=\mu_s(|x|)=|c_1|\chi_{(d_0, d_1)}(s)+\sum_{j=2}^N|c_j|\chi_{[d_{j-1}, d_j)}(s), s\in (0, d_N)
$$ and $\mu_s(x)=\mu_s(|x|)=0, s\in[d_N, 1)$.
Similarly, $$
\mu_s^l(x)=\mu_s^l(|x|)= \sum_{j=1}^N|c_j|\chi_{(d_{j-1}, d_j]}(s), s\in (0, d_N]
$$ and $\mu_s^l(x)=\mu_s^l(|x|)=0, s\in(d_N, 1)$.
Let $f(t)=1-\mu_t(x)$. Then
\begin{eqnarray*}
\mu_t(f)&=&\chi_{(0, 1-d_n)}+\Sigma_{i=1}^N(1-|c_i|)\chi_{[1-d_i, 1-d_{i-1})}\\
&=& 1-\{0\chi_{(0, 1-d_n)}+\Sigma_{i=1}^N |c_i| \chi_{[1-d_i, 1-d_{i-1})}\}\\
&=&1-\mu_{1-t}^l(x),
\end{eqnarray*}
which implies that the first equation holds. The proof of the other containment is similar.
\end{proof}

\begin{lemma}\label{lemma 12}
Let $x\in \mathcal{M}$ be a self-adjoint contractive operator, then
\[\mu_t(\un-|x|)\leq \mu_t(\un-x).\]
\end{lemma}
\begin{proof}
Let $\xi\in \mathcal{H}$, we have
\begin{eqnarray*}
\langle \xi, (\un-x) \xi\rangle=\langle\xi, \xi \rangle-\langle\xi, x\xi \rangle\geq \langle\xi, \xi \rangle-\langle \xi,|x|\xi \rangle=\langle \xi, (\un-|x|) \xi\rangle.
\end{eqnarray*}
By the monotonicity of $\mu_t$ proved by
 \cite[Lemma 2.5]{FK1986}, we immediately get the conclusion.
\end{proof}

\begin{lemma} \label{lemmar2rr}
Let $x,~y\in \mathcal{M}$ be two contractive operators and $r\geq1$. Then
\[\int_0^t\log(1-\mu_s(|xy|^r))ds\geq\int_0^t \log(1-\mu_s(|x|^r|y|^r))ds.\]
Moreover,
\[\int_{1-t}^1\log\mu_s^l(\un- |xy|^r)ds\geq\int_{1-t}^1 \log\mu_s^l(\un-||x|^r|y|^r|)ds.\]
\end{lemma}
\begin{proof}
Let $g(t)=-\log(1-t)$, then $g(x)$ is increasing and convex on $(0,1)$. For $t>0$, set
\[x_t=\frac{x+t\un}{1+2t}~\text{and}~y_t=\frac{y+t\un}{1+2t},\]
then $x_t,~y_t$ are invertible and $\|x_t\|<1,~\|y_t\|<1$.
From the property of polar decomposition of $x_1$ and $x_2$, we obtain $\mu(x_1x_2)=\mu(|x_1||x_2^*|)$.
Using \cite[proposition 2.4]{Han2016},
we get
 $$\int_0^h \mu_s(g(|x_ty_t|^r))ds \leq \int_0^h \mu_s(g(||x_t|^r|y_t|^r|))ds.$$
Since $-\int_0^h \log (1-\mu_t(|x_t y_t|^r))ds\leq -\int_0^h \log(1-\mu_t(x_t^r y_t^r))ds$, we have
$$\int_0^h \log (1-\mu_t(|x_t y_t|^r))ds \geq \int_0^h \log(1-\mu_t(x_t^r y_t^r))ds.$$
Taking limits with $t\rightarrow 0^+$, $x_t\rightarrow x,~y_t\rightarrow y$,
the general case follows that $\mu_s(xy)=\mu_s(x_t y_t)$.
Then we get the results from the Lemma \ref{lemma 10}.
\end{proof}

\begin{lemma}\label{afa proposition 2.4}
Let $r\geq1$ and $f$ be a continuous increasing function on $[0, \infty)$ such that
$f(0)=0$ and
$t\rightarrow f(e^t)$ is convex. For $0\leq x, y\in L_0(\mathcal{M})$, we have
\begin{equation*}
 \int_0^t\mu_s^l(f(|xy|^r))ds\leq \int_0^t\mu_s^l(f(|x^ry^r|))ds.
\end{equation*}
\end{lemma}
\begin{proof}
Replacing left-continuous with right-continuous, $\mu_t^l(|x|)$ have the same property as $\mu_t(x)$.
 By slightly modifying the proof of in \cite[Proposition 2.4]{Han2016}, we can prove the lemma and omit the details.
\end{proof}

\begin{lemma} \label{lemmar3rr}
Let $x,~y\in \mathcal{M}$ be two contractive operators and $r\geq1$. Then
\[\int_0^t\log(1-\mu_s^l(|xy|^r))ds\geq\int_0^t \log(1-\mu_s^l(|x|^r|y|^r))ds.\]
Moreover,
\[\int_{1-t}^1\log\mu_s(\un- |xy|^r)ds\geq\int_{1-t}^1 \log\mu_s(\un-||x|^r|y|^r|)ds.\]
Thus $\Delta(\un- |xy|^r)\leq\Delta(\un-||x|^r|y|^r|).$
\end{lemma}
\begin{proof}
Using Lemma \ref{afa proposition 2.4}, the proof can be done similarly to Lemma \ref{lemmar2rr}.
The details are omitted.
\end{proof}

A generalized H\"{o}lder type eigenvalue inequality (\ref{ineq 1.004}) be extended to the following theorem.

\begin{theorem}\label{mainthm2}
 Suppose $x_1,\ldots,x_m\in \mathcal{M}$ are contractive operators, $r\geq 1$
 and $p_1,\ldots,p_m>0$ with $\sum_{i=1}^m\frac{1}{p_i}=1$.
 Then for all $t>0$, we have
\begin{equation}\label{main2}
\int_0^t \log (1-\mu_s(|x_1\cdots x_m|)^r)ds \geq \sum_{i=1}^m\int_0^t \log (1-\mu_s(|x_i|)^{rp_i})^{\frac{1}{p_i}}ds.
\end{equation}
Moreover,
\begin{equation}\label{main3}
\int_{1-t}^1 \log \mu_s^l(\un- |x_1\cdots x_m|^r)ds \geq \sum_{i=1}^m\int_{1-t}^1 \log \mu_s^l(\un-|x_i|^{rp_i})^{\frac{1}{p_i}}ds.
\end{equation}
\end{theorem}
\begin{proof} Fixed $x,~y\in \mathcal{M}$ and $p,~q>0$ with $\frac{1}{p}+\frac{1}{q}=1$.
By \cite[Theorem 3.3]{FM2005}, we have
$\mu(xy^*)\leq\mu(\frac{1}{p}|x|^p+\frac{1}{q}|x|^q).$
Then $\lambda_t(xy^*)\leq\lambda_t(\frac{1}{p}|x|^p+\frac{1}{q}|x|^q).$
It follows from  \cite[Lemma 3.3]{BH2015} that
there is a unitary operator $U\in\mathcal{M}$  such that
\[U|xy^*|U^*\leq \frac{1}{p} |x|^p+\frac{1}{q}|y|^q+\varepsilon_1 \mathbb{I},\]
for every $\varepsilon_1>0.$
Let $x,~ y\in \mathcal{M}$ be contractive operators. Then
$$\un-|x|^p\geq0,~\un-|y|^q\geq0.$$
Thus there exists a unitary operator $V\in\mathcal{M}$ such that
\begin{eqnarray*}
\un-U|xy^*|U^*&\geq& \frac{1}{p} (\un-|x|^p)+\frac{1}{q}(\un-|y|^q)-\varepsilon_1 \un\\
&\geq & V| (\un-|x|^p)^\frac{1}{p}(I-|y|^q)^\frac{1}{q}|V^*-\varepsilon_2 \un-\varepsilon_1 \un,
\end{eqnarray*}
for a fixed $\varepsilon_2>0.$
By \cite[Lemma 2.4]{DD1992} and the property of rearrangements, we get
\begin{eqnarray*}
\mu_s(\un-|xy^*|)&=&\mu_s(\un-U|xy^*|U^*)\\
&\geq &\mu_s(| (\un-|x|^p)^\frac{1}{p}(\un-|y|^q)^\frac{1}{q}|)-(\varepsilon_2 +\varepsilon_1 )\mu_s(\un)).
\end{eqnarray*}
Letting $\varepsilon_1\rightarrow0$ and  $\varepsilon_2\rightarrow0$, we deduce
\[\mu_s(\un-|xy^*|)\geq \mu_s(| (\un-|x|^p)^\frac{1}{p}(\un-|y|^q)^\frac{1}{q}|).\]

On the other hand,
it is clear that
\begin{eqnarray*}
\tau(\chi_{(s,\infty)}(\un-x))&=&\tau(\chi_{(s,\infty)}(\un-x))\\
&=&\int_s^\infty\chi_{(s,\infty)}(1-t)d E_t(x)\\
&=&\int_0^{1-s}\chi_{(s,\infty)}(t)d E_t(x)+\int_{1+s}^\infty\chi_{(s,\infty)}(t)d E_t(x)\\
&=&\tau(\chi_{[0,1-s)}(x))+\tau(\chi_{(1+s,\infty)}(x))
\end{eqnarray*}
and
\[\tau(\chi_{(s,\infty)}(x)=m\{r>0:\mu_r(x)>s\}.\]

Since $\tau(\un)=1$, it from the normality of $\tau$, we have
\[\tau(\chi_{[s,\infty)}(x)=m\{r\geq 0:\mu_r(x)\geq  s\}.\]

Hence, for $s\geq 0$, we have
\begin{eqnarray*}
\tau(\chi_{(s,\infty)}(\un-x))&=& m\{r\geq 0:0\leq r\leq \tau(\un)~\text{and}~\mu_r(x)<1-s\}\\
&&+m\{r\geq 0:\mu_r(x)> 1+ s\}\\
&=&m\{r\geq 0:\mu_r(\un)-\mu_r(x)>s\}.
\end{eqnarray*}
Consequently, let $f(t)=\mu_t(\un)-\mu_t(|xy|)$, we get
\begin{equation}\label{a1}\mu_s(\un-|xy|)=\mu_s(f(t)).
\end{equation}
Since $\mu_t(\un)-\mu_t(|xy|)=1-\mu_t(|xy|),~0< t< 1$, by Lemma \ref{lemma 10} and Equation (\ref{a1}),
\[\mu_s(\un-|xy|)=1-\mu_{1-s}^l(|xy|),~0< s< 1\] and
\[\mu_s^l(\un-|xy|)=1-\mu_{1-s}(|xy|),~0< s< 1.\]
Combining this with Lemma \ref{lemmar2rr}, we show that
\begin{eqnarray*}
\int_0^t \log(1-\mu_s(|x y|))ds&=&\int_0^t \log(\mu_{1-s}^l(\un-|x y|))ds\\
&\geq &\int_0^t \log(\mu_{1-s}^l(| (\un-|x|^p)^\frac{1}{p}(\un-|y^*|^q)^\frac{1}{q}|))ds\\
&\geq &\int_0^t \log(\mu_{1-s}^l(| (\un-|x|^p)^\frac{1}{p})\mu_{1-s}^l((\un-|y^*|^q)^\frac{1}{q}|))ds\\
&= &\int_0^t \log( (1-\mu_{s}(|x|)^p)^\frac{1}{p}(1-\mu_s(|y^*|)^q)^\frac{1}{q}|)ds.\\
\end{eqnarray*}
Replacing $x$ and $y$ by $|x_1|^r$ and $|x_2|^r$, we have
\[\int_0^t \log(1-\mu_s(|x_1|^r|x_2|^r))ds \geq \int_0^t \log( (1-\mu_{s}(|x_2|^r)^p)^\frac{1}{p}(1-\mu_s(|x_2|^r)^q)^\frac{1}{q}|)ds.\]
Using the property of polar decomposition of $x_1$ and $x_2$, we obtain $\mu(x_1x_2)=\mu(|x_1||x_2^*|)$.
Hence, by Lemma \ref{lemmar2rr} and (\ref{equa 2.1}) and (\ref{pre 1.2}), we deduce
\begin{eqnarray*}\int_0^t \log(1-\mu_s(|x_1x_2|^r))ds
&=&\int_0^t \log(1-\mu_s(|x_1||x_2^*|)^r)ds\\
&\geq &\int_0^t \log(1-\mu_s(|x_1|^r|x_2^*|^r))ds\\
&\geq &\int_0^t \log( (1-\mu_{s}(|x_1|^{rp})^p)^\frac{1}{p}(1-\mu_s(|x_2^*|^{rq})^q)^\frac{1}{q}|)ds\\
&=&\int_0^t \log( (1-\mu_{s}(|x_1|^{rp})^\frac{1}{p}(1-\mu_s(|x_2|^{rq})^\frac{1}{q}|)ds.
\end{eqnarray*}
So, the inequality (\ref{main2}) holds in $m=2$ case.

Suppose the inequality (\ref{main2}) holds for a fixed $m(\geq 2)$ and $r\geq 1$. Let $x_1, \ldots, x_m$, $x_{m+1}\in \mathcal{M}$ are contraction operators, $r\geq 1$ and $p_1, \ldots, p_m$, $p_{m+1}>0$ with $\sum_{i=1}^{m+1}\frac{1}{p_i}=1$. Let $p=\sum_{i=1}^{m}\frac{1}{p_i}$, i.e., $\frac{1}{p}+\frac{1}{p_{m+1}}=1$.
By the above discussion,
we have
\begin{eqnarray*}
&&\int_0^t \log(1-\mu_s(|x_1\cdots x_m x_{m+1}|^r))ds\\
&=&\int_0^t \log( (1-\mu_{s}(|x_1\cdots x_m|^{rp})^\frac{1}{p}(1-\mu_s(|x_{m+1}|^{rp_{m+1}})^\frac{1}{p_{m+1}}|)ds\\
&\geq &\sum_{i=1}^m\int_0^t (\log( (1-\mu_{s}(|x_i|^{rp\cdot\frac{p_i}{p}}))^{\frac{p_i}{p}\cdot\frac{1}{p}}ds+\int_0^t \log((1-\mu_s(|x_{m+1}|^{rp_{m+1}})^\frac{1}{p_{m+1}}|)ds\\
&=&\sum_{i=1}^{m+1}\int_0^t \log((1-\mu_s(|x_{m+1}|^{rp_{m+1}})^\frac{1}{p_{m+1}}|)ds.\\
\end{eqnarray*}
So, the inequality (\ref{main2}) holds for $m+1$ case. That is, (\ref{main2}) holds for every finite $m$.
  Finally, inequality (\ref{main3}) can be obtained by direct calculation and Lemma \ref{lemma 10}.
\end{proof}

\begin{corollary}\label{corollary 13}
Suppose $x_1, x_2, \ldots, x_m\in \mathcal{M}$ are contractive operators and $p_1,\ldots, p_m>0$ satisfies $\sum_{i=1}^m \frac{1}{p_i}=1$.
 Then for all $t>0$, we have
\begin{equation*}\label{main4}
\int_0^t \log (1-\mu_{s}(|x_1\cdots x_m|))ds \geq \sum_{i=1}^m\int_0^t \log (1-\mu_s(|x_i|)^{p_i})^{\frac{1}{p_i}}ds.
\end{equation*}
Moreover, if $x_1\cdots x_m$ is self-adjoint, we have
\begin{eqnarray*}\label{main5}
\int_{1-t}^1\log \mu_s^l(\un- x_1\cdots x_m)ds &\geq&
\int_{1-t}^1\log \mu_s^l(\un- |x_1\cdots x_m|)ds \\
&\geq& \sum_{i=1}^m\int_{1-t}^1\log\mu_s^l(\un - |x_i|^{p_i})^{\frac{1}{p_i}}ds.
\end{eqnarray*}
\end{corollary}

\begin{proof}By the proof of theorem \ref{mainthm2}, we have
\begin{eqnarray*}\label{main6}
\int_0^t \log (1-\mu_{s}(|x_1\cdots x_m|))ds &=&
\int_0^t \log \mu_{1-s}^l(\un-|x_1\cdots x_m|)|ds \\
&=& \int_0^t \log(1-\mu_s(|x_1\cdots x_m|))ds \\
&\geq& \sum_{i=1}^m\int_0^t \log (1-\mu_s(|x_i|)^{p_i})^{\frac{1}{p_i}}ds.\\
\end{eqnarray*}
By Lemma \ref{lemma 12}, we get
\[\mu_s(\un- x_1\cdots x_m)\geq \mu_s(\un- |x_1\cdots x_m|).\]
Then we obtain the results from the Lemma \ref{lemma 10}.
\end{proof}

\begin{remark} From Theorem \ref{mainthm2} and \ref{corollary 13}, we have
$$\Delta(\un-|x_1\cdots x_m|^r)\geq\Pi_{i=1}^m\Delta(\un-|x_i|^{rp_i})^{\frac{1}{p_i}}$$
 and
\begin{eqnarray*}\label{main5}
\Delta(\un-|x_1\cdots x_m|)
\geq \Pi_{i=1}^m\Delta(\un-|x_i|^{p_i})^{\frac{1}{p_i}}.
\end{eqnarray*}
\end{remark}


%
%

\section*{ Funding}
This work is supported by project No. 11801486, No. 11761067
of National Nature Science Foundation of China.

%
%
%
%


\bibliographystyle{amsplain}

\end{document}